\documentclass[12pt]{amsart}
\usepackage{amsmath,amsfonts,amsthm,amscd,textcomp,times, amssymb}
\newtheorem{theorem}{Theorem}

\newtheorem{claim}{Claim}
\newtheorem{proposition}{Proposition}

\newtheorem{definition}{Definition}

\newtheorem{corollary}{Corollary}
\newtheorem{remark}{Remark}
\numberwithin{equation}{section}
\numberwithin{theorem}{section}
\numberwithin{proposition}{section}
\numberwithin{lemma}{section}
\numberwithin{claim}{section}
\numberwithin{corollary}{section}
\newcommand{\bull}{\ensuremath{{}\bullet{}}}

\newcommand{\elam}{\ensuremath{\mathbb{E}_{\lambda_{\bull}}}}
\newcommand{\emu}{\ensuremath{\mathbb{E}_{\mu_{\bull}}}}
 \newcommand{\gr}{\ensuremath{\mathbb{G}(N-n,\mathbb{C}^{N+1})}}
\newcommand{\cpn}{\ensuremath{\mathbb{P}^{N}}}
\newcommand{\slnc}{\ensuremath{SL(N+1,\mathbb{C})}}
\newcommand{\dlb}{\ensuremath{\overline{\partial}}}
\newcommand{\dl}{\ensuremath{\partial}}
\newcommand{\ra}{\ensuremath{\longrightarrow}}
\newcommand{\om}{\ensuremath{\omega}}
\newcommand{\vp}{\ensuremath{\varphi}}
\newcommand{\vps}{\ensuremath{\varphi_{\sigma}}}

\newcommand{\vpt}{\ensuremath{\varphi_{t}}}

\newcommand{\cn}{\ensuremath{\mathbb{C}^{N+1}}}
 \newcommand{\dz}{\ensuremath{\frac{\dl}{\dl z}}} 
\newcommand{\dzb}{\ensuremath{\frac{\dl}{\dl \bar{z}}}} 

\begin{document}
\bibliographystyle{alpha}
 \title{ The discriminant of a space curve is stable} 
\author{Sean Timothy Paul}
\date{June 28 , 2012}
 \vspace{-5mm}
\begin{abstract}{We prove that the discriminant of a nonsingular space curve of genus $g\geq 2$ is stable with respect to the standard action of the special linear group.}
\end{abstract}
\maketitle
 
Historically the construction of the moduli space of curves of genus $g$ was carried out by David Mumford in his epoch making treatise \emph{Geometric Invariant Theory} . The basic first step in the construction was to verify the {stability} of either the Cayley-Chow form or the Hilbert point of a (smooth) complex projective-algebraic curve.\\  
 \ \\
\noindent \textbf{Theorem.} (Mumford \cite{git}) Let $X\ra \cpn$ be a nonsingular space curve of genus $g\geq 2$ embedded by a very ample complete linear system. Assume that $d=\deg(X)\geq 2g$. Let $R_X$ denote the \textbf{\emph{Cayley-Chow form}} of $X$. Then $R_X$ is stable with respect to the action of $G=\slnc$.\newline
\ \\
\noindent \textbf{Theorem.}  (Geiseker \cite{moduli})   Let $X\ra \cpn$ be a nonsingular space curve of genus $g\geq 2$ embedded by a very ample complete linear system. Assume that $d=\deg(X)\geq 2g$. Let $\mathcal{H}_m(X)$ denote the $mth$ \textbf{\emph{Hilbert point}} of $X$. Then there exists a uniform constant $m_0$ such that  $\mathcal{H}_m(X)$ is stable with respect to the action of $G$ for all $m\geq m_0$. \\
\ \\
 The purpose of this note is to establish the corresponding statement for the dual of a nonsingular algebraic curve .
\begin{theorem}\label{main}  Let $X\subset \cpn$ be a nonsingular space curve of genus $g\geq 2$ embedded by a very ample complete linear system. Assume that $d=\deg(X)\geq 2g$. Let $\Delta_X$ denote the \textbf{$X$-discriminant} (the defining polynomial of the dual variety $X^{\vee}$)  . Then $\Delta_X$ is stable with respect to the action of $G$ .  
\end{theorem}
\subsection{Resultants and Discriminants}
Let $X\ra \cpn$ be a complex algebraic curve. Let $R_X$ denote the Cayley-Chow form of $X\ra \cpn$ and let $\Delta_X$ be the $X$-discriminant (see \cite{gkz}) .  For our purpose we must normalize these polynomials
 \begin{align} 
 R := R^{\deg(\Delta_X )}_{X} \ ,\ \Delta :=\Delta ^{\deg(R_X)}_{X} \ .
  \end{align}
  It is well known that up to scaling these lie in finite dimensional (irreducible) $G$-modules $\elam$ and $\emu$ respectively. The reader can consult \cite{paul2011} for details on these modules.
 
 \subsection{The Mabuchi Energy}
 Let $(X,\om)$ be a closed K\"ahler manifold of dimension $n$, where $\om$ denotes the K\"ahler form.
 The space of K\"ahler potentials will be denoted by $\mathcal{H}_\om$
\begin{align*}
\mathcal{H}_\om:=\{\vp\in C^{\infty}(X)\ |\ \om_\vp:=\om+\frac{\sqrt{-1}}{2\pi}\dl\dlb\vp>0 \} \ .
\end{align*}
 
 \begin{definition} (Mabuchi,  \cite{integratingfutaki})
\emph{The \textbf{\emph{K-Energy}} of $(X,\om)$ is the map $\nu_\om:\mathcal{H}_\om\ra \mathbb{R}$ given by the following expression
\begin{align*}
 \qquad \nu_{\omega}(\varphi):= -\frac{1}{V}\int_{0}^{1}\int_{X}\dot{\varphi_{t}}(\mbox{Scal}(\varphi_{t})-\mu)\omega_{t}^{n}dt.
\end{align*}
 $\varphi_{t}$ is a smooth path in $\mathcal{H}_\om$ joining $0$ with $\varphi$, $\mbox{Scal}(\varphi_{t})$ denotes the scalar curvature of the metric $\om_t:= \om+\sqrt{-1}\dl\dlb \varphi_{t}$ , $\mu$ denotes the average of the scalar curvature of $\om$ , and $V$ is the volume. }
\end{definition}
It is well known that the K-Energy does not depend on the path chosen (see \cite{integratingfutaki}) .  
Suppose that $\om$ represents a multiple of the \emph{canonical class}
\begin{align}
\mbox{Ric}(\om)= \frac{\mu}{n}\om+\frac{\sqrt{-1}}{2\pi}\dl\dlb h_\om \ .
\end{align}
In this case there is the following well known direct formula for the K-energy . 
\begin{align} \label{directformula}
\begin{split}
&\nu_{\omega}(\varphi)=\int_{X}\mbox{log}\left(\frac{{\omega_{\varphi}}^{n}}{\omega^{n}}\right)\frac{{\omega_{\varphi}}^{n}}{V} - \frac{\mu}{n}(I_{\omega}(\varphi)-J_{\omega}(\varphi)) -\frac{1}{V}\int_Xh_\om(\om^n_\vp - \om^n) \\
\ \\
& J_{\omega}(\varphi):= \frac{1}{V}\int_{X}\sum_{i=0}^{n-1}\frac{\sqrt{-1}}{2\pi}\frac{i+1}{n+1}\dl\varphi \wedge \dlb
\varphi\wedge \omega^{i}\wedge {\omega_{\varphi} }^{n-i-1}\\
\ \\
&I_{\omega}(\varphi):= \frac{1}{V}\int_{X}\varphi(\omega^{n}-{\omega_{\varphi}}^{n})\ .
\end{split}
\end{align}
 
When $X$ is a compact Riemann surface of genus $g$ (i.e. an algebraic curve) equipped with some K\"ahler metric $\om$ the complicated expression above reduces to 
\begin{align}
\nu_{\omega}(\varphi)=\int_{X}\mbox{log}\left(\frac{{\omega_{\varphi}}}{\omega}\right)\frac{{\omega_{\varphi}}}{V}+\frac{( 2g-2)}{2V^2}\int_X|\nabla \varphi|^2\om-\frac{1}{V}\int_Xh_\om(\om_\vp - \om) \ .
\end{align}

\begin{proposition}\label{lowerbound}\emph{Let $X$ be a smooth algebraic curve of genus $g\geq 1$. Let $\om$ be any K\"ahler metric on $X$. Then
$\nu_{\omega}$ is bounded below on $\mathcal{H}_{\om}$.} 
\end{proposition} 
\begin{proof} The first integral is bounded below by $\frac{-1}{e}$ , the second by $0$, and the third by $-||h_{\om}||_{C^0}$ .
\end{proof}
\begin{remark}
\emph{The Mabuchi energy is bounded below when $g=0$, but this fact is much more difficult to prove.}
\end{remark}

Now suppose that $X \subset \cpn$ is a smooth subvariety (in particular, a projective curve) and that $\om=\om_{FS}|_X$, where $\om_{FS}$ denotes the Fubini-Study K\"ahler form. Given $\sigma\in G$ we have
  \begin{align*}
\sigma^{*}\omega_{FS}= \omega_{FS}+\frac{\sqrt{-1}}{2\pi}\dl\dlb \varphi_{\sigma}>0.
\end{align*}
In this way we produce a set theoretic map  
\begin{align*}
G \ni\sigma\rightarrow \vps \in  \mathcal{H}_{\om}\ .
\end{align*}
It is easy to see that  $\varphi_{\sigma}$ is given by the formula
\begin{align}
\varphi_{\sigma}=\log \frac{||\sigma z||^{2}}{|| z||^{2}}  \ .
\end{align}
We define $\nu_{\om}(\sigma):=\nu_{\om}(\vps)$. In this way \emph{we may consider the K-energy as a function on $G$}. \\
\  \\
\textbf{Theorem A .} (Paul \cite{paul2011})
{ Let $X^n \hookrightarrow \cpn$ be a smooth, linearly normal, complex projective variety of degree $d \geq 2$ . 
 Then there are continuous norms on $\emu$ and $\elam$ such that the Mabuchi energy restricted to $\mathcal{B}$ is given as follows}
\begin{align*} 
\begin{split}
&d^2(n+1)\nu_{\om}({\sigma})=  \log  \frac{{||\sigma\cdot\Delta ||}^{2}}{{||\Delta ||}^{2}} -   \log\frac{{||\sigma\cdot R||}^{2}}{||R||^2}      \ . 
 \end{split}
\end{align*}
 Let $\lambda:\mathbb{C}^*\ra G$ be an algebraic one parameter subgroup of $G$. We shall refer to such maps as \emph{degenerations}.
 
  \begin{definition}
\emph{Let $\mathbb{V}$ be a rational representation of $G$. Let $v\in \mathbb{V}\setminus\{0\}$. Let $\lambda$ be any degeneration in $G$. The \textbf{\emph{weight}}  $w_{\lambda}(v)$ of $\lambda$ on $v$ is the unique integer such that}
\begin{align*}
\lim_{|t|\rightarrow 0}t^{-w_{\lambda}(v)}\lambda(t)v \  \mbox{ {exists in $\mathbb{V}$ and is \textbf{not} zero}}.
\end{align*}
\end{definition}
 The  {Numerical Criterion} of Hilbert and Mumford says that a point $v\in \mathbb{V}$ is (semi)stable if and only if $w_{\lambda}(v)(\leq) < 0$ for all degenerations $\lambda$ in $G$. 
 \begin{proof} (of Theorem \ref{main})  
An immediate application of Theorem A yields the following.
 \begin{proposition}\emph{There is an asymptotic expansion}
 \begin{align*}
  \lim_{|t|\ra 0}  \nu_{\om}(\lambda(t))=\left(w_{\lambda}(\Delta)-w_{\lambda}(R)\right)\log|t|^2+O(1)  \quad  t\in \mathbb{C}^*  \ .
 \end{align*}
\end{proposition}
 
 Proposition \ref{lowerbound} implies that for every degeneration $\lambda$ the following inequality holds
\begin{align}\label{inequality}
w_{\lambda}(\Delta)\leq w_{\lambda}(R) \ .
\end{align}
 When $d\geq 2g$ Mumfords' theorem (and the numerical criterion)  says that $w_{\lambda}(R)<0$. Therefore by (\ref{inequality}) we also have that $w_{\lambda}(\Delta)<0$ and the proof is complete.
 \end{proof}
  Despite the classical nature of the subject the following result seems to be new.
\begin{theorem}\label{plane} Let $X_F\ra \mathbb{P}^2$ be a nonsingular plane algebraic curve with $\deg(F)\geq 2$ . Then the dual curve is semistable with respect to the action of $SL(3,\mathbb{C})$.
\end{theorem}
We have an interesting corollary of Theorem \ref{plane}.
\begin{corollary}
There exist semistable plane curves with cusp singularities.
\end{corollary}
 \bibliography{ref.bib}
\end{document}